\documentclass[12pt]{article}
\usepackage{amsmath}
\usepackage{amsfonts}
\usepackage{pb-diagram}
\usepackage{amsthm} 
\usepackage{amscd}
\newcounter{numb}

\theoremstyle{plain} 
\newtheorem{theorem}{Theorem}[section]
\newtheorem{corollary}[theorem]{Corollary}

\newtheorem{proposition}[theorem]{Proposition}

\newtheorem{remark}[theorem]{Remark}
\numberwithin{equation}{section}

\newcommand{\w}{\wedge}

\newcommand{\p} {\partial}
\newcommand{\g} {\mathfrak {g}}
\newcommand{\kl}{\mathfrak {K}}

\newcommand{\h}{\mathfrak {h}}

\newcommand{\D}{\mathfrak {D}}

\begin{document}
\title{A Relative Theory for Leibniz $n$-Algebras }
\author{Guy Roger Biyogmam}
\date{}
\maketitle{Department of Mathematics,\\Southwestern Oklahoma State University,\\100 Campus Drive,\\Weatherford, OK 73096, USA,\\ \texttt{Email:}\textit{guy.biyogmam@swosu.edu}.}

\begin{abstract}
In this paper we show that for a $n$-Filippov algebra $\g,$ the tensor power $\g^{\otimes n-1}$ is endowed with a structure of anti-symmetric co-representation over the Leibniz algebra $\g^{\w n-1}$. This co-representation is  used to define two relative theories for Leibniz $n$-algebras with $n>2$ and obtain exact sequences relating them. As a result, we construct a spectral sequence for the Leibniz homology of Filippov algebras. 
\end{abstract}

\section{Introduction and Generalities}
The concept of Filippov algebras (also known as $n$-Lie algebras) was first introduced in 1985 by Filippov \cite{F} and was generalized to the concept of Leibniz $n$-algebras by Casas, Loday and Pirashvili \cite{CLP}. Both concepts are of a considerable importance in Nambu Mechanics \cite{N}. Applications of Filippov algebras are found in String theory \cite{B} and in  Yang-Mills theory \cite{FU}.  In section 3, we construct a filtration of a certain complex relative to the complex defined by Casas \cite {C} and  study  the corresponding spectral sequence.\\
Recall that given a field $\kl$ of characteristic different to 2, a Leibniz $n$-algebra \cite{CLP} is defined as a $\kl$ -vector space $\g$ equipped with an $n$-linear operation $[-,\ldots,-]_{\g}: \g^{\otimes n}\longrightarrow \g$ satisfying the identity 
$$[[x_1,\ldots,x_n]_{\g},y_1,\ldots,y_{n-1}]_{\g}=\sum^{n}_{i=1}[x_1,\ldots,x_{i-1},[x_i,y_1,\ldots,y_{n-1}]_{\g},x_{i+1},\ldots,x_n]_{\g}.~~~(1.1)$$
Notice that in the case where the $n$-linear operation $[-,\ldots,-]_{\g}$ is anti-symmetric in each pair of variables, i.e., $$[x_1,x_2,\ldots,x_i,\ldots,x_j,\ldots,x_n]_{\g}=-[x_1,x_2,\ldots,x_j,\ldots,x_i,\ldots,x_n]_{\g}$$ or equivalently $[x_1,x_2,\ldots,x,\ldots,x,\ldots,x_n]_{\g}=0~\mbox{for all $x \in$ G},$ the Leibniz $n$-algebra becomes a Filippov algebra (more precisely a $n$-Filippov algebra). Also, a Leibniz $2$-algebra is exactly a Leibniz algebra \cite[p.326]{L} and become a Lie algebra if the binary operation $[~,~]_{\g}$ is skew symmetric.

 For a Leibniz $n$-algebra $\g,$  It is shown that  \cite[proposition 3.4]{CLP} $D_n(\g)=\g^{\otimes n-1}$ is a Leibniz algebra with the bilinear product $\omega_1:D_n(\g)\times D_n(\g)\rightarrow D_n(\g),$ $$\omega_1(a_1\otimes\ldots\otimes a_{n-1},b_1\otimes\ldots\otimes b_{n-1})=\sum^{n-1}_{i=1}a_1\otimes\ldots\otimes [a_i,b_1,\ldots,b_{n-1}]_{\g}\otimes\ldots\otimes a_{n-1}.$$
Similarly for  a $n$-Filippov algebra $\g$, the bracket $\bar{\omega}_0:L_n(\g)\times L_n(\g)\rightarrow L_n(\g),$ $$\bar{\omega}_0 (a_1\w\ldots\w a_{n-1},b_1\w\ldots\w b_{n-1})=\sum^{n-1}_{i=1}a_1\w\ldots\w [a_i,b_1,\ldots,b_{n-1}]_{\g}\w\ldots\w a_{n-1}$$ provides $L_n(\g)=\g^{\w n-1}$ with a Leibniz algebra structure (see \cite{FO}).

\section{Homology of Leibniz $n$-algebras}
Casas  introduced homology of Leibniz $n$-algebras $_{n}HL_*(\g,\kl)$ with trivial coefficients \cite{C}. In this paper we introduce a spectral sequence as a tool to compute this homology.
 Recall from \cite{LP} that for a Leibniz algebra $\h$ over a ring $\kl$, a co-representation  of $\h$ is a $\kl$-module $M$ equipped with two actions of $\h$, $$[-,-]:\h\times M\rightarrow M ~~\mbox{and}~~[-,-]:\h\times M\rightarrow M$$ satisfying the following axioms
\begin{itemize}
\item $[[x,y],m]=[x,[y,m]]-[y,[x,m]]$
\item $[y,[m,x]]=[[y,m],x]-[m,[x,y]]$
\item $[[m,x],y]=[m,[x,y]]-[[y,m],x].$
\end{itemize} 
A co-representation is called anti-symmetric when $[m,x]=0,~x\in\g,~m\in M.$

\begin{proposition}\label{corg}
Let $\g$ be a Leibniz $n$-algebra (resp. $n$-Filippov algebra). Then $\g$ is a co-representation  over the leibniz algebra  $D_n(\g)=\g^{\otimes (n-1)}$ (resp.  $L_n(\g)$). Morever the corresponding actions can be extended to $D_n(\g)^{\otimes k}$ and $L_n(\g)^{\otimes k}$ respectively.
\end{proposition}

\begin{proof}
 The actions for $D_n(\g)$ are obtained by considering the trivial co-representation of  $\g$  in \cite[proposition 3.1]{CJ}.  The maps are precisely\\

$\omega_0^l(-,-): D_n(\g)\otimes\g\longrightarrow\g$ defined by $\omega_0^l(x_1\otimes\ldots\otimes x_{n-1},x):=-[x,x_1,\ldots,x_{n-1}].$

$\omega_0^r(-,-): \g\otimes D_n(\g)\longrightarrow\g$ defined by $\omega_0^r(x,x_1\otimes\ldots\otimes x_{n-1}):=[x,x_1,\ldots,x_{n-1}].$\\

The actions for $L_n(\g)$ are similar. 

\end{proof}

Recall  \cite[p.328]{L} that if $\h$ is a  Leibniz algebra and $A$ is a co-representation of $\h$, then the  Leibniz homology of $\h$ with coefficients in $A$ written $HL_*(\h;~A),$ is the homology of the chain complex   $CL_*(\h,A):=A\otimes T^*(\h);$ namely $$A\stackrel{d}{\longleftarrow} A\otimes \h^{\otimes^1}\stackrel{d}{\longleftarrow} A\otimes \h^{\otimes^2}\stackrel{d}{\longleftarrow}\ldots\stackrel{d}{\longleftarrow} A\otimes\h^{\otimes^{k-1}}\stackrel{d}{\longleftarrow} A\otimes\h^{\otimes^k}\leftarrow\ldots$$ where $\h^{\otimes^k}$ is the $n$th fold tensor power of $\h$ over $\kl,$ and where 
$$d(v\otimes h_1\otimes\ldots\otimes h_k)=\sum_{1\leq j\leq k}(-1)^j[v,~h_j]\otimes h_1\otimes\ldots \widehat{h_j}\ldots\otimes h_n$$ $$ +\sum_{1\leq i<j\leq k}(-1)^{j+1}v\otimes h_1\otimes\ldots\otimes h_{j-1}\otimes [h_i,~h_j]\ldots\widehat{h_j}\ldots\otimes h_k.$$ 

For a Leibniz $n$-algebra $\g$ over the field $\kl,$ $D_n(\g)$ acts on $\g$ via $\omega_0$ and  Casas defined the chain complex $$_{n}CL_*(\g):=CL_*(D_{n}(\g),\g),$$ and then defined the homology with trivial coefficients for Leibniz $n$-algebras as the homology of this complex;  i.e. $$_{n}HL_*(\g;~\kl):= HL_*(D_{n}(\g);~\g)~~\cite{C}.$$ 

Also for a Filippov algebra $\g$ over the field $\kl,$ one defines $_{n}H^{Lie}_*(\g,~\kl)$ as the homology of the complex $_nC_*(\g):=CL_*(L_{n}(\g),\g)$ i.e.,  $$_{n}H^{Lie}_*(\g;~\kl):=HL_*(L_{n}(\g);~\g)~~\cite{A}.$$

\begin{proposition}\label{corep}
Let $\g$ be a $n$-filippov algebra. The $n$-ary operation $[-,\ldots,-]_{\g}$ induces an anti-symmetric co-representation   $D_n(\g)$  over the Leibniz algebra $L_n(\g).$ 

\end{proposition}

\begin{proof}
The right action $\omega_2^r(-,-): D_n(\g)\otimes L_n(\g)\longrightarrow D_n(\g) $  is trivial, and the left action is given by the  map  $\omega_2^l(-,-): L_n(\g)\otimes D_n(\g)\longrightarrow D_n(\g) $ defined by 
 $$\omega_2^l(a_1\w\ldots\w a_{n-1},b_1\otimes\ldots\otimes b_{n-1})=-\sum^{n-1}_{i=1}b_1\otimes\ldots\otimes [b_i,a_1,\ldots,a_{n-1}]_{\g}\otimes\ldots\otimes b_{n-1}.$$
Since the right action is trivial, we just need to show  for  $x_i, y_i, b_i,\in\g, i=1\ldots n-1$ and setting $\w_{i=1}^{n-1}x_i:=x_1\w \ldots\w x_{n-1},$  $\w_{i=1}^{n-1}y_i:=y_1\w\ldots\w y_{n-1},$  $\otimes_{i=1}^{n-1}b_i:= b_1\otimes\ldots\otimes b_{n-1},$ that  
$$
\begin{aligned}
 \omega^l_2\big(\bar{\omega}_0(\w_{i=1}^{n-1}x_i,\w_{i=1}^{n-1}y_i),\otimes_{i=1}^{n-1}b_i\big)&=\omega_2^l\big(\w_{i=1}^{n-1}x_i, \omega^l_2(\w_{i=1}^{n-1}y_i, \otimes_{i=1}^{n-1}b_i)\big)\\&~~~~-\omega_2^l\big(\w_{i=1}^{n-1}y_i, \omega^l_2(\w_{i=1}^{n-1}x_i, \otimes_{i=1}^{n-1}b_i)\big).
\end{aligned}
$$
On one hand, we have by definition of $\omega^l_2,$ $\bar{\omega_0}$ and the identity $(1.1)$
$$\omega^l_2\big(\bar{\omega}_0(\w_{i=1}^{n-1}x_i,\w_{i=1}^{n-1}y_i),\otimes_{i=1}^{n-1}b_i\big)=\omega_2^l\big((\sum_{i=1}^{n-1}x_1\w\ldots\w [x_i,y_1,\ldots,y_{n-1}]_{\g}\w\ldots\w x_{n-1}),~\otimes_{i=1}^{n-1}b_i\big)$$ $$= \sum_{i=1}^{n-1}\omega_2^l\big((x_1\w\ldots\w [x_i,y_1,\ldots,y_{n-1}]_{\g}\w\ldots\w x_{n-1}),~\otimes_{i=1}^{n-1}b_i\big)~~~~~~~$$ $$~~~~~~~~=\sum_{i=1}^{n-1}\big(-\sum_{j=1}^{n-1} b_1\otimes\ldots\otimes[
b_j,~x_1,\ldots,[x_i,y_1,\ldots,y_{n-1}],\ldots,x_{n-1}]_{\g}\otimes\ldots b_{n-1}\big)$$$$=\big(-\sum_{j=1}^{n-1} b_1\otimes\ldots\otimes\big[[
b_j,~x_1,\ldots,x_{n-1}]_{\g},y_1,\ldots,y_{n-1}\big]\otimes\ldots \otimes b_{n-1}\big)~~$$$$~~~~~~~~~~~~~~~+\big(\sum_{j=1}^{n-1} b_1\otimes\ldots\otimes\big[[
b_j,~y_1,\ldots,y_{n-1}]_{\g},x_1,\ldots,x_{n-1}\big]\otimes\ldots \otimes b_{n-1}\big).$$

 On the other hand,

$$\omega_2^l\big(\w_{i=1}^{n-1}x_i, \omega^l_2(\w_{i=1}^{n-1}y_i, \otimes_{i=1}^{n-1}b_i)\big)=  \omega_2^l\big(\w_{i=1}^{n-1}x_i,-\sum_{i=1}^{n-1}b_1\otimes\ldots\otimes [b_i,y_1,\ldots,y_{n-1}]_{\g}\otimes\ldots\otimes b_{n-1}\big)$$ $$=  -\sum_{i=1}^{n-1}\omega_2^l\big(\w_{i=1}^{n-1}x_i,~ b_1\otimes\ldots\otimes [b_i,y_1,\ldots,y_{n-1}]_{\g}\otimes\ldots\otimes b_{n-1}\big)~~~~~~~~~~~~~~~~~~~~$$$$~~~=\sum_{i=1}^{n-1}\big(\sum^{n-1}_{j=1, j\neq i} b_1\otimes\ldots\otimes[b_j,x_1,\ldots,x_{n-1}]_{\g}\otimes\ldots\otimes [b_i,y_1,\ldots,y_{n-1}]_{\g}\otimes\ldots\otimes b_{n-1}\big)$$ $$ +\sum_{i=1}^{n-1} b_1\otimes\ldots\otimes\big[[
b_i,~y_1,\ldots,y_{n-1}]_{\g},x_1,\ldots,x_{n-1}\big]\otimes\ldots \otimes b_{n-1}~~~$$
and 

 $$\omega_2^l(\w_{i=1}^{n-1}y_i, \omega^l_2((\w_{i=1}^{n-1}x_i, \otimes_{i=1}^{n-1}b_i))=  \omega_2^l(\w_{i=1}^{n-1}y_i,-\sum_{i=1}^{n-1}b_1\otimes\ldots\otimes [b_i,x_1,\ldots,x_{n-1}]_{\g}\otimes\ldots\otimes b_{n-1})$$ $$=  -\sum_{i=1}^{n-1}\omega_2^l(\w_{i=1}^{n-1}y_i,~ b_1\otimes\ldots\otimes [b_i,x_1,\ldots,x_{n-1}]_{\g}\otimes\ldots\otimes b_{n-1})~~~~~~~~~~~~~~~~~~~~$$$$~~~=\sum_{i=1}^{n-1}(\sum^{n-1}_{j=1, j\neq i} b_1\otimes\ldots\otimes[b_j,y_1,\ldots,y_{n-1}]_{\g}\otimes\ldots\otimes [b_i,x_1,\ldots,x_{n-1}]_{\g}\otimes\ldots\otimes b_{n-1})$$ $$ +\sum_{i=1}^{n-1} b_1\otimes\ldots\otimes\big[[
b_i,~x_1,\ldots,x_{n-1}]_{\g},y_1,\ldots,y_{n-1}\big]\otimes\ldots \otimes b_{n-1}.$$
Hence the equality holds.

\end{proof}

\begin{proposition}\label{corL}
Let $\h$ be a Leibniz algebra and let $M$ be a co-representation of $\h.$ Then $M\otimes \h$ is a co-representation over $\h.$ 
\end{proposition}
\begin{proof}
This is \cite[proposition 3.1]{CJ} in the case $n=2.$ 
\end{proof}

The following proposition is the main result of this section. For simplicity, we use  the notation $(\alpha,\beta)$ for $\alpha\otimes\beta.$

\begin{proposition}\label{corh}
Let $\g$ be a Leibniz $n$-algebra. Then $\g\otimes D_n(\g)$ is a co-representation over the Leibniz algebra $D_n(\g)$ with respect to the actions $\bar{\omega}^l: D_n(\g)\otimes(\g\otimes D_n(\g))\rightarrow \g\otimes D_n(\g),$
 defined by $$\bar{\omega}^l(g,(b_0,b))=(\omega_0^l(g,b_0), b)-(b_0,\omega_1(b,g)),~~b_0\in\g,~b,g\in D_n(\g)$$ and $\bar{\omega}^r: (\g\otimes D_n(\g))\otimes D_n(\g)\rightarrow \g\otimes D_n(\g),$
 defined by $$\bar{\omega}^r((b_0,b),g)=(b_0,\omega_1(b,g))-(\omega_0^r(b_0,g), b),~~b_0\in\g,~b,g\in D_n(\g).$$ Moreover, this action is compatible with $d$ (the Loday boundary map of the complex $_nCL_*(\g)$) and the induced action on $_{n}HL_*(\g;~\kl)$ is trivial. 
\end{proposition}

\begin{proof}
By proposition \ref{corg}, $\g$ is a co-representation over $D_n(\g).$ Now take $\h=D_n(\g)$ and $M=\g$ in proposition \ref{corL} and get the actions from the proof of \cite[proposition 3.1]{CJ}. For the compatibility of these actions with $d,$ we provide a proof for the left action $\bar{\omega}_l.$ The proof for $\bar{\omega}_r$ is similar. Let $g=(a_1,a_2,\ldots,a_n)\in D_n(\g)$ and $\alpha=(b_0,b_1,\ldots,b_k)\in\g\otimes D_n(\g)^{\otimes k}$ with $b_i=(x_1^i,\ldots, x_{n-1}^i).$ We claim that  $$d(\bar{\omega}^l_k(g,\alpha))=\bar{\omega}_{k-1}^l(g,d(\alpha)),$$ where $\bar{\omega}^l_k$ is the linear extension of the action $\bar{\omega}^l$ on $\g\otimes\D_n(\g)^{\otimes k}.$  We proceed by induction on $k.$ Indeed, for $k=0$ the result is clear as $$d(\bar{\omega}_0(g,\alpha))=d(\omega_0^l(g,b_0))=d(-[b_0,a_1,\ldots,a_{n-1}]_{\g})=0=\bar{\omega}^l_0(g,0)=\bar{\omega}_0^l(g,d(\alpha)).$$ 
For $k=1,$ we have on one hand 
$$
\begin{aligned}
\bar{\omega}_1^l(g,d(\alpha))&=\bar{\omega}_1^l(g,d(b_0\otimes b_1))\\&=\bar{\omega}_1^l(g,-[b_0,x_1^1,\ldots,x_{n-1}^1]_{\g})\\&=\bar{\omega}_0^l(g,-[b_0,x^1_1,\ldots,x_{n-1}^1]_{\g})\\&=\big[[b_0,x_1^1,\ldots,x_{n-1}^1]_{\g},a_1,\ldots,a_{n-1}\big]_{\g}.
\end{aligned}
$$
On the other hand,
$$
\begin{aligned}
d(\bar{\omega}_1^l(g,\alpha))&=d\big((\omega_0^l(g,b_0), b_1)-(b_0,\omega_1(b_1,g))\big)
\\&=d\big([b_0,a_1,\ldots,a_{n-1}]_{\g}, b_1\big)-\sum_{i=1}^{n-1}d\big(b_0\otimes x_1^1\otimes\ldots\otimes[x_i^1,a_1,\ldots,a_{n-1}]_{\g}\otimes\ldots\otimes x_{n-1}^1)\\&
=\big[[b_0,a_1,\ldots,a_{n-1}]_{\g},x_1^1,\ldots,x_{n-1}^1\big]_{\g}+\sum_{i=1}^{n-1}\big[b_0, x_1^1,\ldots,[x_i^1,a_1,\ldots,a_{n-1}]_{\g},\ldots, x_{n-1}^1\big]_{\g}.
\end{aligned}
$$ Therefore the result holds by the identity (1.1) above.
Now assume that the result holds for $k,$ and write $\alpha=(\beta,b_{k+1})$ for $\alpha=(b_0,b_1,\ldots,b_{k+1}).$ Then we have the following by definition of the Loday boundary map $d$ : $$d(\alpha)=d(\beta,b_{k+1})=(d(\beta),b_{k+1})+(-1)^{k+1}\bar{\omega}_k^l(b_{k+1},\beta)~~~~~~~~\textbf{(2.1.0)}.$$  So we have  
$$
\begin{aligned}
d(\bar{\omega}_{k+1}^l(g,\alpha))&=d\big(\bar{\omega}_{k+1}^l(g,(\beta,b_{k+1}))\big)\\&=d\big[(\bar{\omega}_k^l(g,\beta),b_{k+1})-\big(\beta,\omega_1(b_{k+1},g)\big)\big]~~\mbox{by definition of}~\bar{\omega}_{k+1}^l\\&
=\big(d(\bar{\omega}^l_k(g,\beta)),b_{k+1}\big)+(-1)^{k+1}\omega^l_k(b_{k+1},\bar{\omega}^l_k(g,\beta))-\big(d(\beta),\omega_1(b_{k+1},g)\big)\\&~~~~-(-1)^{k+1}\bar{\omega}^l_k\big(\omega_1(b_{k+1},g),\beta\big)~~~\mbox{by}~\textbf{(2.1.0)}\\&=\big(\bar{\omega}_{k-1}^l(g,d(\beta)),b_{k+1}\big)-\big(d(\beta),\omega_1(b_{k+1},g)\big)\\&~~~~-(-1)^{k+1}\big[\bar{\omega}^l_k(\omega_1(b_{k+1},g),\beta) -\omega_k^l(b_{k+1},\bar{\omega}_k^l(g,\beta))\big]~~\mbox{by inductive hypothesis}\\&
=\bar{\omega}^l_k\big(g,(d(\beta),b_{k+1})\big)+(-1)^{k+1}\omega_k^l\big(g,\bar{\omega}^l_k(b_{k+1},\beta)\big)~~\mbox{by}~\textbf{(1.1)}\\&
=\bar{\omega}_{k}^l(g,d(\beta,b_{k+1})\big)\\&=\bar{\omega}^l_{k}(g,d(\alpha)) ~~\mbox{by}~\textbf{(2.1.0)}.
\end{aligned}
$$
Now we prove that the induced action on $_{n}HL_*(\g;~\kl)$ is trivial. Indeed, it is clear from \textbf{(2.1.0)} that for a fix $g\in D_n(\g),$ the endomorphisms $\bar{\omega}_k^l(g,-)$ and $d(-,g)-(d(-),g)$ of $\g\otimes D_n(\g)$ satisfy the identity $$\bar{\omega}^l_k(g,-)=(-1)^{k+1}[d(-,g)-(d(-),g)].$$ Therefore if $\beta$ is a cycle in $_{n}CL_k(\g),$ we have $$\bar{\omega}^l_k(g,\beta)=(-1)^{k+1}[d(\beta,g)-(d(\beta),g)]=(-1)^{k+1}d(\beta,g)$$ which is zero in homology.

\end{proof}

\section{A spectral sequence for Filippov algebras}
In this section, $\g$ is a $n$-Filippov algebra ($n>2$) of dimension at least $n$ over the field $\kl.$
The canonical projection $$\pi_1^k: \g\otimes D_n(\g)^{\otimes (k+1)}\longrightarrow \g\otimes L_n(\g)^{\otimes (k+1)},~k\geq 0$$ induces a chain map, $_{n}CL_*(\g)\longrightarrow  ~_nC_*(\g),$ and thus  a $\kl -$linear map on homology 
$$_{n}HL_*(\g;\kl)\longrightarrow~ _{n}H^{Lie}_*(\g;~\kl).$$

Now let $$_{n}C_k^{rel}(\g):=ker(\pi_1^k).$$ Clearly $_{n}C_*^{rel}(\g)$ is a subcomplex of $_{n}CL_*(\g).$ We define  the relative theory $_{n}H^{rel}_*(\g)$ as the homology of  $_{n}C_*^{rel}(\g).$ It results the following long exact sequence relating the homologies of $n$-Lie algebras and Leibniz $n$-algebras:
\begin{proposition}
There is a long exact sequence\\ $\ldots\stackrel{\p}{\longrightarrow}~_{n}H^{rel}_{k-1}(\g)\longrightarrow~ _{n}HL_{k}(\g;~\kl)\longrightarrow~_{n}H^{Lie}_{k}(\g;~\kl)\stackrel{\p}{\longrightarrow} ~_{n}H^{rel}_{k-2}(\g)\longrightarrow$\\ $~~~~~~~~~~~~~~~~~~~~\vdots~~~~~~~~~~~~~~~~\vdots~~~~~~~~~~~~~~~~\vdots$\\ $\ldots\stackrel{\p}{\longrightarrow}~_{n}H^{rel}_{0}(\g)\longrightarrow ~_{n}HL_{1}(\g;~\kl)\longrightarrow~_{n}H^{Lie}_{1}(\g;~\kl)\stackrel{\p}{\longrightarrow} 0$\\ $~ ~~~~~~~~~~~~~~~~~0\longrightarrow~_{n}HL_{0}(\g;~\kl)\stackrel{\cong}{\longrightarrow} ~_{n}H^{Lie}_{0}(\g;~\kl)\stackrel{\p}{\longrightarrow} 0.$
\end{proposition}

\begin{proof}
$$0\longrightarrow ~_{n}C^{rel}_*(\g)\longrightarrow ~_{n}CL_*(\g)\longrightarrow ~_nC_*(\g)\longrightarrow 0,$$ is a short exact sequence of chain complexes.
\end{proof}

Also, the canonical projection $$\pi_2^k: D_n(\g)\otimes L_n(\g)^{\otimes k}\longrightarrow  L_n(\g)^{\otimes k+1},~k\geq 0$$ induces a chain map, $CL_*(L_n(\g),D_n(\g))\longrightarrow CL_*(L_n(\g),\kl),$ and thus a $\kl-$linear map on homology 
$$HL_*(L_n(\g),D_n(\g))\longrightarrow~ HL_*(L_n(\g),~\kl)$$ (the co-representation $D_n(\g)$ over $L_n(\g)$ is given in proposition \ref{corep}).\\
Now let $$_{n}DR_k(\g):=ker(\pi_2^{k+1}).$$ Then $_{n}DR_*(\g)$ is a subcomplex of $CL_*(L_n(\g),D_n(\g)).$ We define  the relative theory $_{n}HD_*(\g)$ as the homology of that complex. It results the following :

\begin{proposition}
There is a long  exact sequence\\ $\ldots\stackrel{\p}{\longrightarrow}~_{n}HD_{k-1}(\g)\longrightarrow~ HL_{k}(L_n(\g);~D_n(\g))\longrightarrow~HL_{k+1}(L_n(\g);~\kl)\stackrel{\p}{\longrightarrow} ~_{n}HD_{k-2}(\g)\longrightarrow$\\ $~~~~~~~~~~~~~~~~~~~~\vdots~~~~~~~~~~~~~~~~\vdots~~~~~~~~~~~~~~~~\vdots$\\ $\ldots\stackrel{\p}{\longrightarrow}~_{n}HD_{0}(\g)\longrightarrow ~HL_{1}(L_n(\g);~D_n(\g))\longrightarrow~HL_{2}(L_n(\g);~\kl)\stackrel{\p}{\longrightarrow} 0$\\ $~ ~~~~~~~~~~~~~~~~~0\longrightarrow~HL_{0}(L_n(\g);~D_n(\g))\stackrel{\cong}{\longrightarrow} ~HL_{1}(L_n(\g);~\kl)\stackrel{\p}{\longrightarrow} 0.$
\end{proposition}

\begin{proof}
$$0\longrightarrow ~ _{n}DR_*(\g)\longrightarrow CL_*(L_n(\g),D_n(\g))\longrightarrow CL_*(L_n(\g),\kl)\longrightarrow 0$$ is a short exact sequence of chain complexes.
\end{proof}

\begin{remark}
 $\g\otimes D_n(\g)$ is a co-representation over the Leibniz algebra $L_n(\g)$ (the proof is similar to proposition \ref{corh}) and the canonical projection $$\pi_3^k:\g\otimes D_n(\g)\otimes L_n(\g)^{\otimes k}\longrightarrow \g\otimes L_n(\g)^{\otimes k+1}$$ induces a chain map, $CL_*(L_n(\g),~\g\otimes D_n(\g))\longrightarrow CL_*(L_n(\g),~\g),$ and thus a $\kl-$linear map on homology $$_{n}H^{Lie}_{*}(\g;~ D_n(\g))\longrightarrow _{n}H_{*+1}^{Lie}(\g;~\kl).$$ 
 Set $\Gamma_{\g}=ker[D_n(\g)\longrightarrow L_n(\g)].$ Note that $\g$ is a co-representation over $\Gamma_{\g}$ by means of the restrictions of the actions $\omega_0^l$ and $\omega_0^r,$ and $\Gamma_{\g}$ is a co-representation over  $L_n(\g)$  by means of the restrictions of the actions $\omega_2^l$ and $\omega_2^r.$ Also, $~ker[\pi_3^{*}]=\g\otimes \Gamma_{\g}\otimes L_n(\g)^{\otimes *}$ is a subcomplex of $CL_*(L_n(\g),~\g\otimes D_n(\g)).$

The projection $\pi^k_1$ can be written as the composition of projections $$\g\otimes D_n(\g)^{\otimes k+1}\longrightarrow \g\otimes D_n(\g)\otimes L_n(\g)^{\otimes k}\longrightarrow \g\otimes L_n(\g)^{\otimes k+1}$$ which leads to a natural map between exact sequences

\[ 
\begin{diagram} 
  \node{_{n}H^{rel}_{k-1}(\g)} \arrow{e,t}{} 
  \node{_{n}HL_{k}(\g;~\kl)} \arrow{e,t}{} 
  \node{_{n}H^{Lie}_{k}(\g;~\kl)} \arrow{e,t}{}  
  \node{_{n}H^{rel}_{k-2}(\g)} 
  \node{}\\ 
  \node{_{n}H_{k}^{Lie}(\g;~\Gamma_{\g})} \arrow{e,t}{} \arrow{n,r}{}
  \node{_{n}H^{Lie}_{k}(\g;~D_n(\g))} \arrow{e,t}{} \arrow{n,r}{}
  \node{_{n}H_{k+1}^{Lie}(\g;~\kl)} \arrow{e,t}{} \arrow{n,r}{}
  \node{_{n}H^{Lie}_{k+1}(\g;~\Gamma_{\g})}                \arrow{n,r}{}
  \node{} \\
  \node{}
\end{diagram}
\] 
where $_{n}H^{Lie}_{*}(\g,~\Gamma_{\g})$ stands for the homology of the complex $ker[\pi_3^{*}].$

\end{remark}
The following theorem is the main result of this paper.
\begin{theorem}
Let $\g$ be a $n$-Filippov algebra (with $n>2$) of dimension at least  $n$. Then there exist a spectral sequence converging to $_{n}H^{rel}_*(\g)$ with $$E^2_{r,s}\cong~ _{n}HL_s(\g;~\kl)\otimes ~_{n}HR_r(\g),~~r\geq0,~s\geq0.$$
\end{theorem}
\begin{proof}
 We consider the filtration of the complex $$_{n}C_k^{rel}(\g)=ker[\g\otimes D_n(\g)^{\otimes  k+1}\longrightarrow \g\otimes L_n(\g)^{\otimes k+1}],~k\geq 0$$ given by $$F^r_s:=\g\otimes D_n(\g)^{\otimes s}\otimes ker[D(\g)^{\otimes r+1}\longrightarrow L_n(\g)^{\otimes r+1}]~~r\geq0,~s\geq0.$$ Let $a\in\g,$ $u=(b_1,\ldots,b_{s})\in D_n(\g)^{\otimes s}$ and $v=(x_1,\ldots,x,\ldots,x,\ldots,x_{n-1}) \in \Gamma_g.$ For simplicity, we use again the notation $(\alpha,\beta)$ for $\alpha\otimes\beta.$ Denote by $W_1$ the extension of the action $\omega_1$ on $D_n(\g)^{\otimes s}.$ Then as $\omega_0^r(a,v)=0$ and $$W_1(u,v)=\sum_{i=1}^{s} (b_1,b_2,\ldots,\omega_1(b_i,v),\ldots,b_s)=0, $$ it follows that $d(a,u,v)=(d(a,u),v).$ So $F_*^0=\g\otimes D_n(\g)^{\otimes *}\otimes \Gamma_{\g}$ is a subcomplex of $_{n}C^{rel}_*(\g),$ and we have  $$H_*(F_*^0)\cong ~_{n}HL_*(\g;~\kl)\otimes\Gamma_{\g}.$$ 
Now keep $a,$ $u$ and $v$ as above and let $\mu=(e_1,\ldots,e_i),$ $\mu_i=(e_1,\ldots,\hat{e}_j,\ldots,e_i),~1\leq j\leq i$ with $e_i \in D_n(\g).$ We have \\

$d(a,u,v,\mu)=(d(a,u),v,\mu)+(-1)^s(a,u,v,d\mu)$\\

$~~~~~~~~~~~~~~~~~+\sum_{j=1}^i(-1)^{j+s}(\omega_0^r(a,e_j),u,v,\mu_j)+\sum_{j=1}^{i}(-1)^{j+s}(a,W_1(u,e_j),v,\mu_j).$\\

Note that in the above calculations, we used the fact that $$W_1(u,e_j)=\sum_{t=1}^q(b_1,\ldots,b_{t-1},\omega_1(b_t,e_j),\dots,b_{q})$$ and $\omega_0^l(v,y)=0=\omega_0^r(y,v)$ for all $y\in\g$ and for all $v\in \Gamma_{\g}.$  

It is now clear that $F^0_*\subset F^1_*\subset F^2_*\subset\ldots$ are subcomplexes of $_{n}C^{rel}_*(\g)$ with $F_0^k=~_{n}C^{rel}_k(\g).$

Now   $$E^0_{r,s}=\frac{F^r_s}{F^{r-1}_{s+1}}=\frac{\g\otimes D_n(\g)^{\otimes s}\otimes ker[D(\g)^{\otimes r+1}\longrightarrow L_n(\g)^{\otimes r+1}]}{\g\otimes D_n(\g)^{\otimes {s+1}}\otimes ker[D(\g)^{\otimes r}\longrightarrow L_n(\g)^{\otimes r}]}$$

$$\cong\g\otimes D_n(\g)^{\otimes s}\otimes\frac{ker[D(\g)^{\otimes r+1}\longrightarrow L_n(\g)^{\otimes r+1}]}{D_n(\g)\otimes ker[D(\g)^{\otimes r}\longrightarrow L_n(\g)^{\otimes r}]}$$

$$\cong\g\otimes D_n(\g)^{\otimes s}\otimes ker[D_n(\g)\otimes L_n(\g)^{\otimes r}\longrightarrow L_n(\g)^{\otimes r+1}]$$\\

$~~~~~~~~~~~~~~~~~~~~\cong ~_{n}CL_s(\g,\kl)\otimes~_{n}DR_r(\g).$\\

Thus the corresponding $E^1_{*,*}$ of the spectral sequence has the form $$E^1_{r,s}=H_s(E^0_{r,*})\cong ~_{n}HL_s(\g;~\kl)\otimes ~_{n}DR_r(\g).$$ 

Finally, an examination of the Loday boundary map $d$ above together with the fact that the actions of $D_n(\g)$ and $L_n(\g)$ on $_{n}HL_s(\g;~\kl)$ are trivial by proposition \ref{corh}, lead to 

$$E^2_{r,s}=H_r(E^1_{*,s})\cong ~_{n}HL_s(\g;~\kl)\otimes~_{n}HD_r(\g).$$
\end{proof}

\begin{remark}
For a Leibniz algebra $\g,$ the complexes $_{2}DR_*(\g)$ and $_{2}C^{rel}_*(\g)$ are trivial. So the spectral sequence above does not generalize the Pirashvili spectral sequence \cite{TP}.
\end{remark}

\begin{corollary}
If $\g$ is a $n$-Filippov algebra with $_{n}H_*^{Lie}(\g;~\kl)=0,$ then $_{n}HL_*(\g;~\kl)=0.$
\end{corollary}

\begin{proof}
It is clear from proposition 3.1 that $_{n}HL_0(\g;~\kl)=0.$ So $E_{r,0}^2 =0$ implying that $_{n}H^{rel}_0(\g)=0.$ Thus $_{n}HL_1(\g;~\kl)=0$ from Proposition 3.1. So $E_{0,1}^\infty = E_{1,0}^\infty =0,$ thus $_{n}H^{rel}_1(\g)=0.$ By induction; if $_{n}H^{rel}_k(\g)=0$ for all $k\leq m,$ then proposition 3.1 implies that $_{n}HL_k(\g;~\kl)=0$ for all $k\leq m+1.$ Therefore we have from theorem 3.4 that $E_{r,s}^2=0$ for $s\leq m+1.$ Thus $_{n}H^{rel}_k(\g)=0$ for all $k\leq m+1.$ Hence $_{n}H^{rel}_k(\g)=0$ for all $k\geq 0.$ We use proposition 3.1 again to obtain the result.
\end{proof}


\end{document}